\documentclass[reqno]{amsart}
\usepackage{amssymb,amsmath,enumerate,extarrows}
\usepackage[colorlinks=true,linkcolor=black,citecolor=black]{hyperref}
\usepackage[all]{xy}

\newcommand{\mm}{\mathfrak m}
\newcommand{\nn}{\mathfrak n}

\newcommand{\Z}{\mathbb{Z}}
\newcommand{\N}{\mathbb{N}}
\newcommand{\Q}{\mathbb{Q}}
\newcommand{\Fc}{\mathcal{F}}

\DeclareMathOperator{\chara}{char}
\DeclareMathOperator{\codim}{codim}
\DeclareMathOperator{\den}{den}

\DeclareMathOperator{\embdim}{embdim}
\DeclareMathOperator{\gr}{gr}
\DeclareMathOperator{\Hilb}{HF}
\DeclareMathOperator{\Img}{Im}
\DeclareMathOperator{\id}{id}
\DeclareMathOperator{\Ker}{Ker}
\DeclareMathOperator{\lind}{ld}
\DeclareMathOperator{\linp}{lin}

\DeclareMathOperator{\reg}{reg}

\DeclareMathOperator{\Tor}{Tor}

\newtheorem{thm}{\bf Theorem}[section]
\newtheorem{lem}[thm]{\bf Lemma}

\newtheorem{cor}[thm]{\bf Corollary}
\newtheorem{prop}[thm]{\bf Proposition}

\theoremstyle{definition}

\theoremstyle{remark}
\newtheorem{quest}[thm]{\bf Question}
\newtheorem{rem}[thm]{Remark}
\newtheorem{ex}[thm]{Example}

\theoremstyle{remark}
\newtheorem*{rem*}{Remark}

\numberwithin{equation}{section}
\title[Absolutely Koszul algebras]{The absolutely Koszul property of Veronese subrings and Segre products}

\author{Hop D. Nguyen}
\address{Dipartimento di Matematica, Universit\`{a} di Genova, Via Dodecaneso 35, 16146 Genova, Italy}
\email{ngdhop@gmail.com}


\begin{document}

\begin{abstract}
Absolutely Koszul algebras are a class of rings over which any finite graded module has a rational Poincar\'e series. We provide a criterion to detect non-absolutely Koszul rings. Combining the criterion with machine computations, we identify large families of Veronese subrings and Segre products of polynomial rings which are not absolutely Koszul. In particular, we classify completely the absolutely Koszul algebras among Segre products of polynomial rings, at least in characteristic $0$.
\end{abstract}

\maketitle

\section{Introduction}

Let $k$ be a field, $R$ be a standard graded $k$-algebra. Let $M$ be a finitely generated graded $R$-module. An important problem in the theory of free resolutions concerns the rationality of the Poincar\'e series 
$$
P^R_M(t)=\sum_{i=0}^\infty \beta_i(M)t^i\in \Q[[t]],
$$ 
where $\beta_i(M)=\dim_k \Tor^R_i(k,M)$ denotes the $i$-th Betti number of $M$. While it is well-known that rationality does not hold in general, there are large classes of rings where every module has a rational Poincar\'e series, for example complete intersections \cite{AGP}, \cite{Gu} and Golod rings \cite{Les} and \cite[Theorem 5.3.2]{Avr2}. For a remarkable example, assume that $\chara k=0$ and $R=k[x_1,\ldots,x_n]/I^s$ where $I$ is a non-zero proper homogeneous ideal of $k[x_1,\ldots,x_n]$ and $s\ge 2$. Herzog and Huneke \cite{HeHu} proved recently that any such $R$ is a Golod ring, in particular any finitely generated graded module over $R$ has rational Poincar\'e series.

Following Roos \cite{Roos05}, we say that a standard graded $k$-algebra $R$ is {\it good} if there exists a polynomial $\den(t) \in \Z[t]$ such that for every finitely generated graded module $M$, 
$$
P^R_M(t)\cdot \den(t) \in \Z[t].
$$
Otherwise we say $R$ is {\it bad}. We say that $R$ {\it has the Backelin--Roos property} if there exist a complete intersection $Q$ and a Golod homomorphism $Q\to R$ (see \cite{Avr,Avr2, Le} for more details on Golod homomorphisms). By a result attributed to Levin \cite[Proposition 5.18]{AKM} which was also proved by Backelin and Roos \cite{BR}, if $R$ has the Backelin--Roos property then it is good. Proving rationality of Poincar\'e series using the Backelin--Roos property is a powerful method; see, e.g., \cite{AIS,AKM,HS,RS}. Note however that combining \cite[pp.~291--292]{KP} and \cite[Example 5.6]{Kus}, we have examples of (non-Koszul) good rings which do not have the Backelin--Roos property. 

A different method, among others, of establishing the rationality of Poincar\'e series comes from work of Herzog and Iyengar \cite{HIy}. They proved that $P^R_M(t)$ is a rational function with constant denominator, if $M$ has finite {\it linearity defect} (see Section \ref{sect_background}). Moreover, from \cite[Theorem 5.9]{HIy}, if $R$ is a Koszul algebra having the Backelin--Roos property, then any finitely generated module has finite linearity defect. 

Following the terminology in \cite{IyR}, we call $R$ an {\it absolutely Koszul} algebra if every finitely generated graded $R$-module has finite linearity defect. Hence any absolutely Koszul algebra is good. So far, there was no example of a good Koszul algebra which is not absolutely Koszul, and no example of an absolutely Koszul algebra which does not have the Backelin--Roos property. The following basic question about absolutely Koszul rings was raised in \cite[Section 4.2, Question 14]{CDR}.
\begin{quest}
\label{quest_VeroneseSegre}
How do absolutely Koszul algebras behave with respect to algebra operations like taking Veronese subrings, Segre products or fibre products?
\end{quest}
The question is non-trivial since while Koszul algebras behave well with respect to standard operations \cite{BF}, \cite{BM}, it is not clear at all if the same thing happens to absolutely Koszul algebras. In fact, Question \ref{quest_VeroneseSegre} has a negative answer for tensor product: For $R=k[x,y]/(x,y)^2$, $R\otimes_k R$ is not even good, not to mention absolutely Koszul \cite[Theorem 2.4]{Roos05}. For fibre product, the answer to Question \ref{quest_VeroneseSegre} is yes: A fibre product over $k$ of two standard graded $k$-algebras is absolutely Koszul if and only if both factors are so (\cite[Theorem 2.6]{CINR}). The cases of Veronese subring and Segre product require more work, and so far only partial answers are available.
\begin{prop}[{\cite[Corollary 5.4]{CINR}}]
\label{prop_Veronese_positive}
Assume $\chara k=0$. Let $k[x_1,\ldots,x_n]$ be a polynomial ring in $n$ variables and $c\ge 2$ an integer.
The $c$-th Veronese subring $k[x_1,\ldots,x_n]^{(c)}$ of $k[x_1,\ldots,x_n]$ is absolutely Koszul in the following cases:
\begin{enumerate}[\quad\rm(1)]
\item $c=2$ and $n\le 6$;
\item $c\in \{3,4\}$ and $n\le 4$;
\item $c\ge 5$ and $n\le 3$.
\end{enumerate}
\end{prop}
\begin{prop}[{\cite[Proposition 5.9]{CINR}}]
\label{prop_Segre_positive}
Assume $\chara k =0$ and $1\le m \le n$ are integers. The Segre product $S_{m,n}$ of the polynomial rings $k[x_1,\ldots,x_m]$ and $k[y_1,\ldots,y_n]$ is absolutely Koszul when 
\begin{enumerate}[\quad\rm(1)]
\item $m \le 2$,
\item $m=3$ and $n \le 5$,
\item $m=n=4$.
\end{enumerate}
\end{prop}
Some evidence suggests that not all Veronese subrings and Segre products of polynomial rings are absolutely Koszul. It is proved in \cite[Lemma 5.7]{CINR} that for $(n,c)$ equals  either (7,2), (5,3), (5,4), or $(4,c)$, with $c\ge 5$, $k[x_1,\ldots,x_n]^{(c)}$ does not have the Backelin--Roos property. The same thing happens for $S_{m,n}$ if $(m,n)$ equals either (3,6) or (4,5). Our study in this paper complements Propositions \ref{prop_Veronese_positive} and \ref{prop_Segre_positive} by providing classes of Veronese rings and Segre products which are not absolutely Koszul.
\newpage
\begin{thm}
\label{thm_Veronese_negative}
Assume $\chara k=0$. Let $n\ge 1$ and $c\ge 2$ be integers. Then the Veronese ring $k[x_1,\ldots,x_n]^{(c)}$ is \emph{not} absolutely Koszul in the following cases:
\begin{enumerate}[\quad\rm(1)]
\item $c=2$ and $n\ge 7$;
\item $c\in \{3,4\}$ and $n\ge 5$;
\item $c= 5$ and $n\ge 4$.
\end{enumerate}
\end{thm}
\begin{thm}
\label{thm_Segre_negative}
Assume $\chara k =0$ and $1\le m\le n$ be integers. The Segre product $S_{m,n}$ of the polynomial rings $k[x_1,\ldots,x_m]$ and $k[y_1,\ldots,y_n]$ is \emph{not} absolutely Koszul in the following cases 
\begin{enumerate}[\quad\rm(1)]
\item $m=3$ and $n \ge 6$,
\item $m\ge 4$ and $n \ge 5$.
\end{enumerate}
In particular, the only absolutely Koszul $S_{m,n}$ are provided by Proposition \ref{prop_Segre_positive}.
\end{thm}
We do not know whether the Veronese ring $k[x_1,\ldots,x_n]^{(c)}$ where $c\ge 6$ and $n\ge 4$ is absolutely Koszul or not. In our opinion, studying further such Veronese rings is necessary for a complete answer to Question \ref{quest_VeroneseSegre} and is an interesting challenge itself.  

Our proofs of the main results base on a criterion for non-absolutely Koszul rings; see Lemma \ref{lem_conjpair_modules} and Corollary \ref{cor_conjpair}. The idea of the criterion is simple. Assume that $R$ is a standard graded $k$-algebra, and $l_1$ is a linear form. Then the ideal $M=(l_1)$ has positive linearity defect if its first syzygy module, namely $(0:l_1)(-1)$ is not generated in degree 2, namely $(0):l_1$ is not generated by linear forms. Roughly speaking, Corollary \ref{cor_conjpair} says that $M$ has infinite linearity defect if, moreover, the second syzygy module of $M$ has a direct summand which is isomorphic to a shifted copy of $M$: the non-linearity of the first syzygy module will ``propagate" throughout the minimal free resolution of $M$. Corollary \ref{cor_conjpair} is loosely spoken ``opposite'' to a criterion for absolutely Koszul rings via exact pairs of zero divisors due to Henriques and \c{S}ega \cite{HS}; see Remark \ref{rem_exactzerodivisors}. In order to use our criterion for non-absolutely Koszul rings for proving Theorems \ref{thm_Veronese_negative} and \ref{thm_Segre_negative}, we rely on extensive search with Macaulay2 \cite{GS}; the strategy of our search is described at the beginning of Section \ref{sect_Segre}. Such costly computations are necessary to compensate for the non-constructive nature of the criterion, i.e. its silence on how to construct the linear form $l_1$. It would be interesting to seek for more conceptual proofs of the main results of this paper. 

The paper is organized as follows. After a background section, we present the afore-mentioned criterion for non-absolutely Koszul rings in Section \ref{sect_criterion}. Applications of this criterion to the non-absolute Koszulness of certain Segre products and Veronese subrings are presented in Sections \ref{sect_Segre} and \ref{sect_Veronese}. We close the paper by discussing some open problems. 

\section{Background}
\label{sect_background}

We assume familiarity with standard terminology and knowledge of commutative algebra, for which the books \cite{BH} and \cite{Eis} serve as good references.
\subsection{Linearity defect}
We recall the notion of linearity defect which Herzog and Iyengar introduced in \cite{HIy}, based in the notion of the linear part of a minimal free resolution. The linear part appeared in work of Herzog et al.~\cite[Section 5]{HSV} and Eisenbud et al.~\cite{EFS}. Let $(R,\mm,k)$ be a noetherian local ring, and $M$ a finitely generated $R$-module. Denote $\gr_{\mm} R=\oplus_{i\ge 0}\mm^i/\mm^{i+1}$ the associated graded ring of $R$ with respect to the $\mm$-adic filtration. Let the minimal free resolution of $M$ be
\[
F: \cdots \xrightarrow{\partial} F_i \xrightarrow{\partial} F_{i-1} \xrightarrow{\partial} \cdots \xrightarrow{\partial} F_1 \xrightarrow{\partial} F_0 \to 0.
\]
Since $F$ is minimal, it admits a filtration $\cdots \subseteq \Fc^iF \subseteq \Fc^{i-1}F\subseteq \cdots \subseteq \Fc^0F=F$ as follows:
\[
\Fc^iF: \cdots \to  F_j \to F_{j-1} \to \cdots \to F_i \to \mm F_{i-1} \to \cdots \to \mm^{i-1}F_1 \to \mm^iF_0 \to 0.
\]
The associated graded complex 
$$
\linp^R F=\bigoplus_{i\ge 0} \frac{\Fc^iF}{\Fc^{i+1}F}
$$ 
is called the linear part of $F$. Clearly $\linp^R F$ is a complex of graded $\gr_{\mm}R$-modules with 
\[
(\linp^R F)_i=(\gr_{\mm}F_i)(-i).
\]
The {\it linearity defect} of $M$ is 
\[
\lind_R M=\sup\{i: H_i(\linp^R F)\neq 0\}.
\]
By convention, if $M=0$, we set $\lind_R M=0$. When $\lind_R M=0$, we say that $M$ is a {\it Koszul module}. With appropriate changes, the above constructions also apply if $R$ is a standard graded $k$-algebra with the graded maximal ideal $\mm$, and $M$ a finitely generated graded $R$-module. 

The linearity defect can be characterized in terms of maps between Tor modules.
\begin{thm}[\c{S}ega, {\cite[Theorem 2.2]{Se}}]
\label{thm_Sega}
Let $d\ge 0$ be an integer. The following are equivalent:
\begin{enumerate}[\quad \rm (1)]
\item $\lind_R M\le d$,
\item The map $\Tor^R_i(R/\mm^{q+1},M)\to \Tor^R_i(R/\mm^q,M)$ induced by the canonical surjection $R/\mm^{q+1}\to R/\mm^q$ is zero for all $i>d$ and all $q\ge 0$.
\end{enumerate}
\end{thm}
\begin{rem}
A substantial simplification of Theorem \ref{thm_Sega} for $R$ being a standard graded $k$-algebra and $M$ a finitely generated graded $R$-module was discovered by Katth\"an \cite[Theorem 1.3]{K}.
\end{rem}
Using Theorem \ref{thm_Sega}, we can prove the following useful statements.
\begin{prop}[{\cite[Corollary 2.10]{Ng}} and {\cite[Proposition 4.3]{NgV}}]
\label{prop_lindef}
Let $0\to M\xrightarrow{\phi} P\to N \to 0$ be a short exact sequence of finitely generated $R$-modules.
\begin{enumerate}[\quad \rm (1)]
\item If $P$ is free then $\lind_R M=\lind_R N-1$ if $\lind_R N\ge 1$ and $\lind_R M=0$ if $\lind_R N=0$.
\item Assume that $\Tor^R_i(k,\phi)=0$ for all $i\ge 0$. Then there are inequalities:
\begin{align*}
\lind_R M &\le \max\{\lind_R P, \lind_R N-1\},\\
\lind_R P &\le \max\{\lind_R M, \lind_R N\},\\
\lind_R N &\le \max\{\lind_R M+1,\lind_R P\}.
\end{align*}
\end{enumerate}
\end{prop}
\subsection{Koszul algebras and regularity}
Let $R$ be a standard graded $k$-algebra. Namely $R$ is an $\N$-graded commutative algebra with $R_0=k$ such that $R$ is generated by finitely many elements of degree 1 over $R_0$. Let $M$ be a finitely generated graded $R$-module. The {\it regularity} of $M$ over $R$ is 
\[
\reg_R M=\sup\{j-i: \Tor^R_i(k,M)_j\neq 0\}.
\]
We say that $M$ has a {\it linear resolution} over $R$ if there exists $d\in\Z$ such that $\Tor^R_i(k,M)_j=0$ for all $j\neq i+d$. (We also say that $M$ has a $d$-linear resolution in that case.) Modules with linear resolutions are typical examples of Koszul modules. 

We say that $R$ is a {\it Koszul algebra} if the residue field $k$ has a linear resolution over $R$. By R\"omer's \cite[Theorem 3.2.8]{Ro}, which extends Yanagawa's \cite[Proposition 4.9]{Yan}, if $R$ is a Koszul algebra, then $M$ is a Koszul module if and only if for each $d\in \Z$, the submodule $M_{\left<d\right>}$ of $M$ generated by elements of degree $d$ has a $d$-linear resolution. 

A collection $\Fc$ of ideals generated by linear forms of $R$ is a {\it Koszul filtration} \cite{CTV} if the following conditions are satisfied:
\begin{enumerate}
\item $(0),\mm \in \Fc$.
\item For every $I\neq (0)$ in $\Fc$, there exist an ideal $J\in \Fc$ properly contained in $I$ such that $I/J$ is a cyclic module and $J:I\in \Fc$.
\end{enumerate}
If $R$ has a Koszul filtration $\Fc$, then any ideal $I$ of $\Fc$ has a linear resolution over $R$. In particular, $R$ is a Koszul algebra. 

It is useful to recall the following folkloric criterion for a module to have a linear resolution. Let $R$ be a standard graded $k$-algebra, and $M$ a finitely generated graded module. We say that $M$ has {\it linear quotients} if there exist minimal homogeneous generators $m_1,\ldots,m_s$ of $M$ such that for all $i=0,1,\ldots,s-1$, the ideal $(m_1,\ldots,m_{i-1}):m_i$ has a $1$-linear resolution over $R$. 
\begin{lem}
\label{lem_linearquotient}
Let $R$ be a standard graded $k$-algebra, and $M$ a finitely generated graded module with linear quotients. Assume that $M$ is generated by elements of the same degree $d$. Then $M$ has a $d$-linear resolution.
\end{lem}
\begin{proof}
Assume that $M$ is minimally generated by $m_1,\ldots,m_s$, all of degree $d$, and $I_j=(m_1,\ldots,m_{j-1}):m_j$ has a $1$-linear resolution for all $j=1,\ldots,s$. An easy induction shows that $(m_1,\ldots,m_j)$ has $d$-linear resolution for all $j=1,\ldots,s$.
\end{proof}

\subsection{Absolutely Koszul rings}
Let $(R,\mm,k)$ be a noetherian local ring with the unique maximal ideal $\mm$, or a standard graded $k$-algebra with the graded maximal ideal $\mm$. Herzog and Iyengar \cite[Proposition 1.8]{HIy} pro\-ved that if $M$ is a finitely generated (graded) $R$-module such that $\lind_R M<\infty$, then $P^R_M(t)$ is a rational function with constant denominator (depending only on $R$).
We say that $R$ is {\it absolutely Koszul} if $\lind_R M<\infty$ for every finitely generated (graded) $R$-module $M$. If $R$ is a graded absolutely Koszul algebra, then it is Koszul and good in the sense of Roos \cite{Roos05}. 

For the remaining of this section, we will restrict ourselves to the graded case. The following two results are useful to reduce the embedding dimension when proving that certain ring is not absolutely Koszul. The first was stated for local rings but its graded analog is immediate.
\begin{thm}[{\cite[Theorem 5.2]{Ng}}]
\label{thm_Koszulmap}
Let $R\to S$ be a surjection of standard graded $k$-algebras such that $S$ has a linear resolution as an $R$-module. Let $N$ be a finitely generated graded $S$-module. Then there is an equality $\lind_S N=\lind_R N$. In particular, if $R$ is an absolutely ring then so is $S$.
\end{thm}
Recall that a homomorphism of standard graded $k$-algebras $S\xrightarrow{\theta} R$ is called an {\it algebra retract} if there is a homomorphism of graded $k$-algebras $R\xrightarrow{\phi} S$ such that $ \phi \circ \theta=\id_S$. 
\begin{lem}
\label{lem_retract}
Let $S\xrightarrow{\theta} R$ be an algebra retract of standard graded $k$-algebras. If $R$ is absolutely Koszul then so is $S$. 
\end{lem}
We say that $R$ has {\it the Backelin--Roos property} if there exists a Golod map of standard graded $k$-algebras $Q\to R$ where the defining ideal of $Q$ is generated by a regular sequence (see \cite{Avr, Le}). If $R$ is Koszul, then by \cite[Proposition 5.8]{HIy}, $R$ has the Backelin--Roos property if and only if there exists a map $Q\xrightarrow{\phi} R$, where $Q$ is defined by a regular sequence of quadratic polynomials, and $\Ker \phi$  has $2$-linear resolution over $Q$. By \cite[Theorem 5.9]{HIy}, if $R$ is Koszul with the Backelin--Roos property, then $R$ is absolutely Koszul.

The Hilbert series of $R$ is 
$$
HS_R(t)=\sum_{i=0}^{\infty}(\dim_k R_i)t^i\in \Q[[t]].
$$
By the Hilbert-Serre's theorem, we can write
$$
HS_R(t)=\frac{h_R(t)}{(1-t)^{\dim R}},
$$
where $h_R(t)\in \Z[t]$ is called the $h$-polynomial of $R$. There is an obstruction on the Hilbert function of Koszul algebras with the Backelin--Roos property.
\begin{prop}[{\cite[Proposition 3.12, Corollary 3.13]{CINR}}]
\label{prop_HSobstruction}
Assume that $R$ is a Koszul algebra with the Backelin--Roos property. Denote by $\codim R$ the codimension of $R$. Then the formal power series
\[
1- \frac{h_R(-t)}{(1-t)^{\codim R}}
\]
has only non-negative coefficients. 

Write $h_R(t)=(1+t)^sg(t)$, where $s\ge 0$ and $g(t)\in \Z[t]$ is such that $g(-1)\neq 0$. If $R$ is furthermore not a complete intersection, then $g(-1)<0$.
\end{prop}
\section{A criterion for non-absolutely Koszul rings}
\label{sect_criterion}
Let $(R,\mm,k)$ be a noetherian local ring. Let $M$ be a non-trivial finitely generated $R$-module and $M_1,M_2$ be non-trivial submodules such that $M=M_1+M_2$. Following \cite{FHV}, we say that the decomposition $M=M_1+M_2$ is a {\em Betti splitting} if for all $i\ge 0$, there is an equality $\beta_i(M)=\beta_i(M_1)+\beta_i(M_2)+\beta_{i-1}(M_1\cap M_2)$. Using the long exact sequence of Tor associated with the short exact sequence
\[
0\to M_1\cap M_2 \to M_1 \oplus M_2 \to M \to 0,
\]
we can easily prove the following are equivalent: 
\begin{enumerate}
\item $M=M_1+M_2$ is a Betti splitting;
\item For all $i\ge 0$, the natural maps $\Tor^R_i(k,M_1\cap M_2) \to \Tor^R_i(k, M_1)$ and $\Tor^R_i(k,M_1\cap M_2) \to \Tor^R_i(k, M_2)$ are zero.
\end{enumerate}
In the graded case, we define Betti splittings with suitable modifications. 

We record the following simple lemma, which is useful to establish graded Betti splittings.
\begin{lem}
\label{lem_zeromap}
Let $R$ be a standard graded $k$-algebra. Let $M\xrightarrow{\phi} N$ be morphism of finitely generated graded $R$-modules. Assume that $M_i=0$ for all $i\le \reg_R N$. Then $\Tor^R_i(k,M)\xrightarrow{\Tor^R_i(k,\phi)} \Tor^R_i(k,N)$ is the zero map for all $i\ge 0$.
\end{lem}
\begin{proof}
Denote $d=\reg_R N$. For any $j$, consider the induced map $\Tor^R_i(k,M)_j \to  \Tor^R_i(k,N)_j$. If $j>i+d$, we have $\Tor^R_i(k,N)_j=0$. If $j\le i+d$, since $M$ has no generator of degree less than $d+1$, we have $\Tor^R_i(k,M)_j=0$. The desired conclusion follows.
\end{proof}

The following lemma is our main tool in proving that an algebra is not absolutely Koszul. 
\begin{lem}
\label{lem_conjpair_modules}
Let $(R,\mm,k)$ be a noetherian local ring. Let $p,q\ge 1$ be integers and $\phi_1:R^p\to R^q,\phi_2: R^q\to R^p$ be non-zero maps of free $R$-modules such that $\phi_1\circ \phi_2=0$ and $\phi_2 \circ \phi_1=0$. Assume that the following conditions are satisfied:
\begin{enumerate}[\quad \rm(1)]
\item Both $\Img \phi_1$ and $\Img \phi_2$ are not Koszul modules;
\item There exist non-trivial submodules $K_1 \subseteq \Ker \phi_2,K_2 \subseteq \Ker \phi_1$ satisfying the following two conditions:
\begin{enumerate}[\quad \rm (a)]
\item There are Betti splittings $\Ker \phi_1=\Img \phi_2+K_2$ and $\Ker \phi_2=\Img \phi_1+K_1$,
\item Each of the modules $K_1\cap \Img \phi_1,K_2\cap \Img \phi_2$ is either zero or Koszul.
\end{enumerate}
\end{enumerate}
Then $\lind_R (\Img \phi_1)=\lind_R (\Img \phi_2)=\infty$. In particular, $R$ is not absolutely Koszul.
\end{lem}
\begin{proof}
Denote $M=\Img \phi_1, N=\Img \phi_2$. Since there is an exact sequence
\[
0\longrightarrow \Ker \phi_1 \longrightarrow R^p \longrightarrow M \longrightarrow 0
\]
and $\lind_R M\ge 1$, by Proposition \ref{prop_lindef}(1), $\lind_R \Ker \phi_1=\lind_R M-1$.

Consider the exact sequence
\[
0\longrightarrow K_2 \cap N \longrightarrow K_2 \oplus N \longrightarrow \Ker \phi_1 \longrightarrow 0.
\]
Since the decomposition $\Ker \phi_1=K_2+N$ is a Betti splitting, the map $\Tor^R_i(k,K_2\cap N) \to \Tor^R_i(k,K_2\oplus N)$ is zero for all $i\ge 0$. Using Proposition \ref{prop_lindef}(2), we get
\[
\max \{\lind_R K_2,\lind_R N\} \le \max\{\lind_R (K_2\cap N),\lind_R \Ker \phi_1\}=\lind_R M-1. 
\]
The equality holds since $\lind_R (K_2\cap N)=0$. Therefore $\lind_R N\le \lind_R M-1$. Similarly, $\lind_R M\le \lind_R N-1$. These inequalities cannot hold simultaneously unless $\lind_R M=\lind_R N=\infty$.
\end{proof}

In practice, to prove that a Cohen-Macaulay algebra is not absolutely Koszul, we will usually pass to an artinian reduction using a regular sequence of linear forms, and use the following special case of (the graded analog of) Lemma \ref{lem_conjpair_modules}. 
\begin{cor}
\label{cor_conjpair}
Let $R$ be a standard graded $k$-algebra. Let $l_1,l_2$ be non-zero linear forms such that $l_1l_2=0$. Assume that there exist ideals $(0) \neq K_1,K_2$ not generated by linear forms such that the following are satisfied:
\begin{enumerate}[\quad \rm(1)]
\item $(0):l_1=K_2+(l_2)$ and $(0):l_2=K_1+(l_1)$;
\item The decompositions $(0):l_1=K_2+(l_2)$ and $(0):l_2=K_1+(l_1)$ are Betti splittings, e.g. $K_1\cap (l_1)=K_2\cap (l_2)=(0)$;
\item Each of the modules $K_1\cap (l_1)$ and $K_2\cap (l_2)$ is either zero or Koszul.
\end{enumerate}
Then $\lind_R (l_1)=\lind_R (l_2)=\infty$ and hence $R$ is not absolutely Koszul.
\end{cor}
\begin{proof}
Generally, if $M$ is a finitely generated graded $R$-module, which is generated in a single degree $d$, then $M\cong \gr_{\mm}(M)(-d)$ as graded $R$-modules. Hence from \cite[Proposition 1.5]{HIy}, $\lind_R M=0$ if and only if $M$ has a $d$-linear resolution over $R$. By the assumption (1), we conclude that $\lind_R (l_1)\ge 1$ and $\lind_R (l_2)\ge 1$. Now using the graded analog of Lemma \ref{lem_conjpair_modules} for the maps $R(-1) \xrightarrow{\cdot l_1} R$ and $R(-1) \xrightarrow{\cdot l_2} R$ we get $\lind_R (l_1)=\lind_R (l_2)=\infty$.
\end{proof}
\begin{rem}
Roughly speaking, the hypotheses of Corollary \ref{cor_conjpair} ensure that the homology of the linear part of the free resolution of $(l_1)$ is non-trivial at every homological degree. For simplicity, assume that $K_1\cap (l_1)=K_2\cap (l_2)=(0)$. Let $F$ be the minimal graded free resolution of $(l_1)$. Then as $K_2\subseteq (0):l_1$, and $K_2$ is not generated by linear forms, a minimal first syzygy of $(l_1)$ induces a non-zero element in $H_1(\linp^R F)$. Similar thing happens for $(l_2)$. Since $(0):l_1=K_2 \oplus (l_2)$, the first syzygies of $(l_2)$ is a direct sum of the second syzygies of $(l_1)$. In particular, a minimal second syzygy of $(l_1)$ induces a non-zero element in $H_2(\linp^R F)$. Moreover, the same thing happens for $(l_2)$. Repeating this argument, we see that the linear part of $F$ has non-zero homology at every homological degree. 
\end{rem}
\begin{rem}
\label{rem_exactzerodivisors}
Henriques and \c{S}ega introduced in \cite{HS} the notion of an exact pair of zero divisors. Assume that $R$ is a standard graded $k$-algebra, and $x, y$ are homogeneous elements in $R$ such that $xy=0$. We say that $x$ and $y$ form an {\it exact pair of zero divisors} if $0:x=(y)$ and $0:y=(x)$. If $x$ and $y$ form an exact pair of zero divisors and both are linear forms, then $\lind_R (x)=0$ since $(x)$ has the linear resolution
\[
\cdots \xrightarrow{\cdot x} R(-4) \xrightarrow{\cdot y} R(-3) \xrightarrow{\cdot x} R(-2) \xrightarrow{\cdot y} R(-1) \to (x) \to 0.
\]
Similarly $\lind_R (y)=0$. One of the main results of \cite{HS}, Theorem 3.3, yields a criterion for absolutely Koszul rings using exact pairs of zero divisors. The hypotheses and conclusion of Corollary \ref{cor_conjpair} are in some sense ``opposite'' to that of \cite[Theorem 3.3]{HS}.
\end{rem}

\begin{ex}
Let $R=k[x,y,z,u]/(x^2,xy,y^2,z^2,zu,u^2)$. This was considered by Roos, who proved that $R$ is a bad Koszul algebra, and in particular not absolutely Koszul \cite[Theorem 2.4]{Roos05}. Corollary \ref{cor_conjpair} yields a direct proof of the non-absolute Koszulness of $R$. The study of this example in fact motivated the non-absolute Koszulness criterion and the main results of this paper. 

Denote $l_1=x-z,l_2=x+z$. Then
\begin{align*}
(0):l_1 &=(yu)+(l_2),\\
(0):l_2 &=(yu)+(l_1),\\
(yu)\cap (l_1) & = (yu) \cap (l_2)=(0).
\end{align*}
Hence the conditions of Corollary \ref{cor_conjpair} are fulfilled. In particular, $\lind_R (x-z)= \lind_R (x+z)=\infty$ and $R$ is not absolutely Koszul.
\end{ex}
\begin{ex}
\label{ex_Conca}
The following construction is taken from \cite[Example 3.8]{Con}. Let $Q=k[x,y,z,u]$ and $R=Q/I$ where
$$
I=(x^2+yz,xy-yu,xz,xu,y^2).
$$
 Let $\phi$ be the map $R(-1)^2=Re_1\oplus Re_2\to R^2$ given by the matrix
\[
\left(
\begin{matrix}
-x & y \\
z  & x
\end{matrix}
\right).
\]
We claim that $\lind_R \Img \phi=\infty$, whence $R$ is not absolutely Koszul.

First $R$ is Koszul since $R$ admits the following Koszul filtration (where by abuse of notation, we use $f$ to denote the residue class of $f$):
\[
\left\{(0),(x),(z),(x,z),(x,y),(z,u),(x,z,u),(x,y,z,u)\right\}.
\]
In detail, there are equalities
\begin{align*}
(0):(x) &=(z,u),\\
(0):(z) &= (x),\\
(z):(x,z) &=(x,z,u),\\
(x):(x,y) &=(x,y,z,u),\\
(z): (z,u) &=(x,z),\\
(z,u) :(x,z,u) &=(x,y,z,u),\\
(x,z,u):(x,y,z,u) &=(x,y,z,u).
\end{align*}

Denote $M=\Img \phi$. Now $\Ker \phi$ equals the image of the map $R(-2)^2\oplus R(-3) \to R(-1)^2=Re_1\oplus Re_2$ given by the matrix
\[
\left(
\begin{matrix}
-x & y  & 0\\
z  & x  & u^2
\end{matrix}
\right).
\]
In particular $\phi^2=0$. We can identify $(\Img \phi)(-1)=M(-1)$ with $(-xe_1+ze_2,ye_1+xe_2)$ and letting $K=(u^2e_2)$, we get a decomposition of the first syzygy of $M$: $\Omega^R_1(M)=\Ker \phi=M(-1)+K$. Note that $K\cong (R/(x,y))(-3)$, so $K$ has a $3$-linear resolution. In the same way, let $L=M(-1)\cap K$, then $L=(zu^2e_2) \cong (R/(x,y))(-4)$ has a 4-linear resolution.

We will show that $\reg_R M=2$. Denote $t_i(M)=\sup\{j: \Tor^R_i(k,M)_j\neq 0\}$, then $\reg_R M=\sup_{i\ge 0}\{t_i(M)-i\}$.  Since $t_1(M)=3$, $\reg_R M\ge 2$, thus it suffices to show that $t_i(M)\le i+2$ for all $i\ge 0$. Induct on $i$; the cases $i=0,1$ are clear. Consider the exact sequence
\[
0\longrightarrow L \longrightarrow M(-1) \oplus K \longrightarrow \Omega^R_1(M) \longrightarrow 0. 
\]
For $i\ge 2$, we have
\begin{align*}
t_i(M)=t_{i-1}(\Omega^R_1(M)) &\le \max\{t_{i-1}(M)+1,t_{i-1}(K),t_{i-2}(L)\}\\
                        &=\max\{t_{i-1}(M)+1,i+2\}=i+2.
\end{align*}
In the display, the second equality holds since $K$ has a 3-linear and $L$ has a 4-linear resolution. The last equality holds by the induction hypothesis. Therefore $\reg_R M=2$.

Since $M$ is generated in degree 1 and $R$ is Koszul, we conclude that $M$ is not Koszul. Moreover, by Lemma \ref{lem_zeromap}, the decomposition $\Omega^R_1(M)=M(-1)+K$ is a Betti splitting. Hence by the graded analog of Lemma \ref{lem_conjpair_modules}, $\lind_R M=\infty$.

Notably, while $R$ is not an integral domain (as its socle contains $yz\neq 0$), there exists no pair $l_1$ and $l_2$ of linear forms of $R$ that satisfy the conditions of Corollary \ref{cor_conjpair}. We give a sketch of the argument here. Assume the contrary, that there exist such a pair of linear forms $l_1,l_2$. Denote by $f,g$ the linear forms of $Q=k[x,y,z,u]$ representing the linear forms $l_1,l_2$ of $R$. 

Working in $Q$, since $fg\in I\subseteq (x,y)$, either $f$ or $g$ belongs to $(x,y)$. We can assume that $f\in (x,y)$. By elementary considerations, up to scaling, only the following cases are possible, where $a,b,\alpha \in k,(a,b)\neq (0,0),\alpha \neq 0$:
\begin{enumerate}
 \item $f=y, g=a(x-u)+by$,
 \item $f=x, g=az+bu$,
 \item $f=x+\alpha y$, $g=a\bigl(x-\alpha y+(1/\alpha)z\bigr)+b(\alpha y-u)$.
\end{enumerate}
\end{ex}
Now we return to $R$. In Case (2), $(0):l_2=(x)$ is generated by a linear form, a contradiction. In Case (3), $(0):l_2=(l_1)$, again a contradiction. In Case (1), we must have $a=0$, otherwise $(0):l_2=(l_1)$. Thus it remains to consider the case $(l_1)=(l_2)=(y)$. We have $(0):y=(y,x-u,z^2)$. By assumption, there exists a submodule $K\subseteq (y,x-u,z^2)=U$ such that $U=(y)+K$ is a Betti splitting. Since $\Tor^R_i(k,(y)\cap K) \to \Tor^R_i(k,(y)) \oplus \Tor^R_i(k,K)\to \Tor^R_i(k,U)$ is exact, the map $\Tor^R_i(k,(y))\to \Tor^R_i(k,U)$ is injective for all $i\ge 0$. In particular, for all $i$, 
\[
\beta_i(U)=\beta_i((y))+\beta_i(U/(y)).
\]
But $\beta_1(U)=7 \neq \beta_1((y))+\beta_1(U/(y))=3+5$. This contradiction finishes the proof.

This example shows that the criterion of Lemma \ref{lem_conjpair_modules} is stronger than that of Corollary \ref{cor_conjpair}. Nevertheless for the main applications of this paper, it suffices to invoke Corollary \ref{cor_conjpair}.
\section{Segre products}
\label{sect_Segre}
In the current and next section, we prove the main theorems \ref{thm_Segre_negative} and \ref{thm_Veronese_negative}. Our general strategy in both cases is as follows. Suppose that $R$ is a Koszul Cohen-Macaulay $k$-algebra which is not a complete intersection such that its $h$-polynomial $h_R(t)$ satisfies $h_R(-1)>0$, and we want to show that $R$ is not absolutely Koszul. (This is the cases for the rings considered in Sections \ref{sect_Segre} and \ref{sect_Veronese}). First we go modulo a regular sequence of linear forms to an artinian reduction $S$ of $R$ (at least by extending $k$), then by Theorem \ref{thm_Koszulmap}, it suffices to prove that $S$ is not absolutely Koszul. Note that $S$ still satisfies $h_S(-1)>0$ since the $h$-polynomial does not change. To simplify the problem, we find an ideal $I$ generated by linear forms of $S$ such that $S/I$ has linear resolution over $S$, and $h_{S/I}(-1)>0$. We can ensure $\reg_S (S/I)=0$ by constructing a Koszul filtration of $S$ containing $I$. The condition $h_{S/I}(-1)>0$ and Proposition \ref{prop_HSobstruction} give at least a partial guarantee that $S/I$ is not absolutely Koszul. Then we construct suitable linear forms $l_1$ and $l_2$ in order to apply Corollary \ref{cor_conjpair} to $S/I$ and deduce that it is not absolutely Koszul. Then thanks to Theorem \ref{thm_Koszulmap}, $R$ is also not absolutely Koszul, and we are done. The most tricky part in this strategy is the search for $l_1$ and $l_2$, for which we depend heavily on computer assistance.

The relevant Segre products have lower embedding dimensions, so we treat them before the Veronese rings. Recall that if $A$ and $B$ are standard graded $k$-algebra, then their Segre product $A*B$ is the subalgebra of $A\otimes_k B$ with the graded component $(A*B)_i=A_i\otimes_k B_i$ for all $i\ge 0$.

In the following, for $1\le m\le n$, let $S_{m,n}$ denote the Segre product $k[x_1,\ldots,x_m]*k[y_1,\ldots,y_n]$. The Hilbert function of $S_{m,n}$ is given by
\[
\Hilb(S_{m,n},i)=\dim_k A_i\dim_k B_i=\binom{m+i-1}{m-1}\binom{n+i-1}{n-1}, \quad \quad \text{ for $i\ge 1$}.
\]
In particular, the $h$-polynomial $h_{S_{m,n}}(t)$ of $S_{m,n}$ is
\[
h_{S_{m,n}}(t)=\sum_{i=0}^{m-1} \binom{m-1}{i}\binom{n-1}{i}t^i
\]
and its embedding dimension and Krull dimension are $\embdim(S_{m,n})=mn$ and $\dim S_{m,n}=m+n-1$, respectively.
\begin{prop}
\label{prop_S36}
Assume that $\chara k=0$. Then the Segre product of $k[x_1,x_2,x_3]$ and $k[y_1,y_2,\ldots,y_6]$ is not absolutely Koszul.
\end{prop}
\begin{proof}
The ring $R=S_{3,6}$ is defined by the $2$-minors of the generic matrix
\begin{center}
\[
\left(
\begin{matrix}
z_{11} & z_{12} & z_{13} & z_{14} & z_{15} & z_{16}\\
z_{21} & z_{22} & z_{23} & z_{24} & z_{25} & z_{26}\\
z_{31} & z_{32} & z_{33} & z_{34} & z_{35} & z_{36}
\end{matrix}
\right)
\]
\end{center}
Denote $U=R/J$, where 
\begin{gather*}
J=(z_{11},z_{12}-z_{21},z_{13}-z_{22}-z_{31},z_{14}-z_{23}-z_{32},z_{15}-z_{24}-z_{33},\\
z_{16}-z_{25}-z_{34},z_{26}-z_{35},z_{36})
\end{gather*}
is generated by a regular sequence of length $8=\dim R$. Then $U$ is a quotient ring $B/H$ of $B=k[z_{21},z_{22},\ldots,z_{34},z_{35}]$. The ring $U$ has Hilbert series $1+10t+10t^2$.

Rename the variables of $B$ in the dictionary order: $a_1=z_{21},a_2=z_{22},\ldots,a_{10}=z_{35}$. Then $B=k[a_1,\ldots,a_{10}]$. In particular, $H$ is the ideal of $2$-minors of the matrix
\begin{center}
\[
\left(
\begin{matrix}
0 & a_1 & a_2+a_6 & a_3+a_7 & a_4+a_8 & a_5+a_9\\
a_1 & a_2 & a_3 & a_4 & a_5 & a_{10}\\
a_6 & a_7 & a_8 & a_9 & a_{10} & 0
\end{matrix}
\right)
\]
\end{center}

Denote $l_1=a_1-a_3+a_5-a_7, l_2=a_2+a_6-a_8-a_{10}$. Then we have
\begin{align*}
(0):l_1 &=(l_2,a_3a_9),\\
(0):l_2   &=(l_1,a_4a_{10}),\\
(l_2)\cap (a_3a_9)&=(l_1) \cap (a_4a_{10})=(0).
\end{align*}
By Corollary \ref{cor_conjpair}, $U$ is not absolutely Koszul, and as $\reg_R U=0$, neither is $R$ by Theorem \ref{thm_Koszulmap}.
\end{proof}

\begin{prop}
\label{prop_S45}
Assume that $\chara k=0$. The Segre product of $k[x_1,x_2,x_3,x_4]$ and $k[y_1,y_2,y_3,y_4,y_5]$ is not absolutely Koszul.
\end{prop}
\begin{proof}
The ring $R=S_{4,5}$ is defined by the $2$-minors of the generic matrix
\begin{center}
\[
\left(
\begin{matrix}
z_{11} & z_{12} & z_{13} & z_{14} & z_{15}\\
z_{21} & z_{22} & z_{23} & z_{24} & z_{25}\\
z_{31} & z_{32} & z_{33} & z_{34} & z_{35}\\
z_{41} & z_{42} & z_{43} & z_{44} & z_{45}
\end{matrix}
\right)
\]
\end{center}
Denote $U=R/J$, where 
\begin{gather*}
J=(z_{11},z_{12}-z_{21},z_{13}-z_{22}-z_{31},z_{14}-z_{23}-z_{32}-z_{41},z_{15}-z_{24}-z_{33}-z_{42},\\
z_{25}-z_{34}-z_{43},z_{35}-z_{44},z_{45})
\end{gather*}
is generated by a regular sequence of length $8=\dim R$. Then $U$ is a quotient ring $B/H$ of $$B=k[z_{21},z_{22},\ldots,z_{43},z_{44}].$$
The ring $U$ has Hilbert series $1+12t+18t^2+4t^3$.

Rename the variables of $B$ in the dictionary order: $a_1=z_{21},a_2=z_{22},\ldots,a_{12}=z_{44}$. Then $B=k[a_1,\ldots,a_{12}]$. In particular, $H$ is the ideal of $2$-minors of the matrix
\begin{center}
\[
\left(
\begin{matrix}
0 & a_1 & a_2+a_5 & a_3+a_6+a_9 & a_4+a_7+a_{10}\\
a_1 & a_2 & a_3 & a_4 & a_8+a_{11}\\
a_5 & a_6 & a_7 & a_8 & a_{12}\\
a_9 & a_{10} & a_{11} & a_{12} & 0
\end{matrix}
\right)
\]
\end{center}

\textsf{Step 1}: Let $\mm$ be the graded maximal ideal of $U$. We claim that there exists a Koszul filtration $\Fc$ of $U$ which contains the following ideals: $(0), (a_{12}), \mm.$ 

In detail, let $\Fc$ be the collection of ideals
\begin{gather*}
(0), I_1=(a_{12}), I_2=(a_{12},a_{11}), I_3=(a_{12},a_{11},a_{10}), I_4=(a_{12},a_{11},a_{10},a_9),\\
I_5=(a_{12},a_{11},\ldots,a_8), I_6=(a_{12},a_{11},\ldots,a_8,a_7+a_4),\\
 I_7=(a_{12},a_{11},\ldots,a_7,a_4),I_8=(a_{12},a_{11},\ldots,a_7,a_4,a_6+a_3),\\
 I_9=(a_{12},a_{11},\ldots,a_7,a_6,a_4,a_3),I_{10}=(a_{12},a_{11},\ldots,a_6,a_5+a_2,a_4,a_3),\\ I_{11}=(a_{12},a_{11},\ldots,a_3,a_2),I_{12}=(a_{12},a_{11},\ldots,a_3),\\
I_{13}=(a_{12},a_{11},\ldots,a_6,a_4), I_{14}=(a_{12},a_{11},\ldots,a_4),\mm.
\end{gather*}
Then $\Fc$ is a Koszul filtration since we have the following identities
\begin{align*}
&(0):I_1 =I_6,  &I_1:I_2 =I_8, \quad \quad & I_2:I_3 =I_{10},   & I_3:I_4 =\mm, \\
& I_4:I_5 =I_{14},  &I_5:I_6  =\mm, \quad \quad  &I_6:I_7 =I_{12}, &I_7:I_8 =\mm,\\
& I_8:I_9 =I_{11},  &I_9:I_{10}=\mm, \quad \quad  & I_{10}:I_{11}=\mm,  & I_9:I_{12}=\mm,\\
& I_7:I_{13}= I_{11}, & I_{13}:I_{14}= \mm, \quad \quad & I_{11}:\mm =\mm. & 
\end{align*}

\textsf{Step 2}: Denote $W=U/(a_{12})$. Note that the Hilbert series of $W$ is $1+11t+12t^2+t^3$. Let $\nn$ be the graded maximal ideal of $W$, then $\nn^3=(a_1a_6a_{11})\neq 0$ and $\nn^4=0$. As we have seen, $\reg_U W=0=\reg_R U$. If $R$ was absolutely Koszul, we conclude by Theorem \ref{thm_Koszulmap} that so is $W$. On the other hand, letting $l_1=a_2+a_4-a_6+a_{10}$, $l_2=a_2+2a_3+2a_5+3a_6-a_8+6a_9-a_{10}$, $K=(a_6a_{11})$, we have
\begin{align}
(0):l_1&=K+(l_2), \nonumber\\
(0):l_2&=K+(l_1), \nonumber\\
K\cap (l_1) &=K\cap (l_2)=(a_1a_6a_{11})=\nn^3 \label{eq_intersectKl}.
\end{align}
We claim that:
\begin{enumerate}[\quad \rm (1)]
\item $\reg_W K=2$,
\item  $\reg_W (l_1)=\reg_W (l_2)=2$.
\end{enumerate}

For (1): Since $\Fc$ contains the ideals $(a_{12})$ and $I_{11}=(a_{12},a_{11},\ldots,a_3,a_2)$, and $(a_{12})$ is the minimal non-zero element of $\Fc$, there is an induced Koszul filtration of $W$ which contains $(a_{11},a_{10},\ldots,a_3,a_2)=(0):(a_6a_{11})$. Hence $\reg_W (a_6a_{11})=2$.

For (2): Clearly $\reg_R (l_1)\ge 2$ and $\reg_R (l_2)\ge 2$. For the reverse inequalities, we prove by induction on $i$ that $t_i((l_1))\le i+2$ and $t_i((l_2))\le i+2$ for all $i\ge 0$. The cases $i=0,1$ are clear. Assume that $i\ge 2$.

Denote $L=\nn^3 \cong (W/\nn)(-3)$, which has a $3$-linear resolution. Consider the exact sequence
\[
0\longrightarrow L \longrightarrow K \oplus (l_2) \longrightarrow (0):l_1 \longrightarrow 0.
\]
The module $((0):l_1)(-1)$ being the first syzygy of $(l_1)$, we obtain
\begin{align*}
t_i((l_1))=t_{i-1}((0):l_1)+1 & \le \max\{t_{i-1}(K),t_{i-1}((l_2)),t_{i-2}(L)\}+1 \\
                              & = \max\{i+2,t_{i-1}((l_2))+1\} \\
                              & =i+2.
\end{align*}
In the chain, the second equality holds since $K$ has a $2$-linear resolution, and $L$ has a $3$-linear resolution. The last equality holds by induction hypothesis. Similarly, $t_i((l_2))\le i+2$, and we finish the proof of (2).

From (1), (2) and \eqref{eq_intersectKl}, the decompositions $(0):l_1=K+(l_2)$ and $(0):l_2=K+(l_1)$ are Betti splittings by an application of Lemma \ref{lem_zeromap}. By Corollary \ref{cor_conjpair}, $W$ is not absolutely Koszul. This contradiction implies that $S_{4,5}$ is not absolutely Koszul.
\end{proof}
From Propositions \ref{prop_S36}, \ref{prop_S45}, we easily deduce Theorem \ref{thm_Segre_negative}.
\begin{proof}[Proof of Theorem \ref{thm_Segre_negative}]
Apply Lemma \ref{lem_retract} and the fact that if $m\le m',n\le n'$ then there is an algebra retract $S_{m,n} \to S_{m',n'}$.
\end{proof}

\section{Veronese rings}
\label{sect_Veronese}
In this section, we prove Theorem \ref{thm_Veronese_negative}. Let $S$ be a standard graded $k$-algebra. Denote $S^{(c)}=\oplus_{i=0}^\infty S_{ic}$ the $c$-th Veronese subring of $S$, whose grading is given by $\deg S_{ic}=i$ for all $i$. For $n,c\ge 1$, denote by $V_{n,c}$ the ring $k[x_1,\ldots,x_n]^{(c)}$.

The Hilbert function of $V_{n,c}=k[x_1,\ldots,x_n]^{(c)}$ is given by
\[
\Hilb(V_{n,c},i)=\binom{n+ic-1}{n-1}, \quad \quad \text{for all $i\ge 1$}.
\]
In particular, the $h$-polynomial of $V_{n,c}$ is
\[
h_{V_{n,c}}(t)=\sum_{i=0}^{n-1}h_it^i, \quad \text{where} ~ h_i=\sum_{j= 0}^i(-1)^{i-j}\binom{n-1+jc}{n-1}\binom{n}{i-j}.
\]
The embedding dimension and Krull dimension of $V_{n,c}$ are $\binom{n+c-1}{c}$ and $n$, respectively.

Let $R$ be a standard graded $k$-algebra with the graded maximal ideal $\mm$. Assume that the ideal $\mm$ is minimally generated by the linear forms $a_1,\ldots,a_n$. We say that $R$ is {\it strongly Koszul} (with respect to the sequence $a_1,\ldots,a_n$), if for every $1\le r\le n$ and every sequence of pairwise distinct elements $i_1,i_2,\ldots,i_r$ of $[n]=\{1,\ldots,n\}$, the ideal $(a_{i_1},\ldots,a_{i_{r-1}}):a_{i_r}$ is generated by a subset of $\{a_1,\ldots,a_n\}$. If $R$ is strongly Koszul with respect to the sequence $a_1,\ldots,a_n$, then clearly the collection of ideals generated by subsets of $\{a_1,\ldots,a_n\}$ forms a Koszul filtration for $R$. Furthermore, for any sequence of pairwise distinct elements $i_1,\ldots,i_r$ in $[n]$, the ring $R/(a_{i_1},\ldots,a_{i_r})$ is also strongly Koszul. It is known that any Veronese subring of a polynomial ring is strongly Koszul (\cite[Proposition 2.3]{HHR}).

Recall the statement of Theorem \ref{thm_Veronese_negative}.
\begin{thm}
\label{thm_Veronese_negative2}
Assume $\chara k=0$. Then $V_{n,c}$ is not absolutely Koszul in the following cases:
\begin{enumerate}[\quad\rm(1)]
\item $c=2$ and $n\ge 7$;
\item $c\in \{3,4\}$ and $n\ge 5$;
\item $c= 5$ and $n\ge 4$.
\end{enumerate}
\end{thm}
\begin{proof}
If $m\le n$ then there is an algebra retract $V_{m,c} \to V_{n,c}$. Therefore from Lemma \ref{lem_retract}, we only need to show that $V_{n,c}$ is not absolutely Koszul if $(n,c)$ belongs to the set 
$$
\{(7,2),(5,3),(5,4),(4,5)\}.
$$
Denote $S=k[x_1,\ldots,x_n]$ and $R=V_{n,c}$. Then $R$ is generated as an algebra by the standard $k$-basis of $S_c$. The vector space dimension of $S_c$ is $N=\binom{n+c-1}{c}$. Order the standard $k$-basis of $S_c$ in the degree revlex order as $m_1<\cdots<m_N$, so that $m_1=x_n^c,m_2=x_n^{c-1}x_{n-1},\ldots,m_N=x_1^c$. Clearly $x_n^c,x_{n-1}^c,\ldots,x_1^c$ is a maximal regular sequence of elements of degree 1 of $R$, so the artinian reduction $U=R/(x_n^c,x_{n-1}^c,\ldots,x_1^c)$ of $R$ has embedding dimension $N-n$. We have a surjection with the source a polynomial ring $B=k[a_1,\ldots,a_{N-n}]\twoheadrightarrow U$. One can obtain $U$ with the following Macaulay2 function $\textsf{artinVeroneseSubring}$.

Input: Integers $n,c\ge 1$.

Output: The artinian reduction $U=V_{n,c}/(x_n^c,\ldots,x_1^c)$ of $V_{n,c}$. The variables of the ambient ring of $U$ are ordered as in the description above.

\begin{verbatim}
artinVeroneseSubring=(n,c)->(
      toricR=QQ[x_1..x_n];
      maxtoricR=ideal vars toricR;
      dsc=flatten entries mingens maxtoricR^c;
      dsg={};
      for i from 1 to n do dsg=append(dsg,x_(n+1-i)^c);
      lengthdsc=length dsc;
      S=QQ[y_1..y_lengthdsc];
      axavero=map(toricR,S,dsc);
      veroId=trim preimage(axavero,ideal(dsg));
      U1=prune(S/veroId);
      Am=ambient U1;
      numvar=length flatten entries vars U1;
      newring=QQ[a_1..a_numvar];
      axanew=map(newring,Am,{a_1..a_numvar});
      artinVeroId=trim axanew(ideal U1);
      newring/artinVeroId)
\end{verbatim}

\textsf{Case 1:} $(n,c)=(7,2)$. 

In this case $N=28$, $B=k[a_1,\ldots,a_{21}]$. The Hilbert series of $U$ is $1+21t+35t^2+7t^3$. Denote $W=U/(a_{21})$, $\nn$ the graded maximal ideal of $W$. The Hilbert series $h_W(t)=1+20t+25t^2+2t^3$ of $W$ satisfies $h_W(-1)>0$, hence it at least does not have the Backelin--Roos property by Proposition \ref{prop_HSobstruction}. Let $\nn$ be the graded maximal ideal of $W$, then $\nn^3=(a_1a_{12}a_{19},a_1a_{12}a_{20})$. Since $R$ is a strongly Koszul ring, so is $W$, and $\reg_R W=0$.

Denote $l_1=a_1+a_2-a_8+a_{12}+a_{20}, l_2=a_1+a_2-a_8+a_{12}-a_{20}$ and $K=(a_2+a_{12},a_3,a_7,a_8-a_{12})a_{19}$. Then 
\begin{align}
(0):l_1 & = K+(l_2),\nonumber\\
(0):l_2 &= K+(l_1),\nonumber\\
K\cap (l_1) &=K\cap (l_2)=(a_1a_{12}a_{19}). \label{eq_intersectKl_Veronese72}
\end{align}

We claim that 
\begin{enumerate}[\quad \rm (1)]
\item $\reg_W K=2$,
\item $\reg_W (l_1)=\reg_W (l_2)=2$.
\end{enumerate}

For (1): $K$ has linear quotients since $W$ is strongly Koszul and
\begin{align*}
&(0):(a_3a_{19}) =(a_i: i\in [20]\setminus 7),\\
&(a_3a_{19}):(a_7a_{19})=\nn,\\
&(a_3a_{19},a_7a_{19}):((a_2+a_{12})a_{19})=\nn,\\
&(a_3a_{19},a_7a_{19},(a_2+a_{12})a_{19}):((a_8-a_{12})a_{19})=\nn.
\end{align*}
By Lemma \ref{lem_linearquotient}, $\reg_W K=2$.

For (2): Denote $L=K\cap (l_1)=(a_1a_{12}a_{19}) \cong (W/\nn)(-3)$. Then $L$ has a $3$-linear resolution. Similarly to Step 2 in the proof of Proposition \ref{prop_S45}, we get $\reg_R (l_1)=\reg_R (l_2)=2.$

Now from (1), (2) and \eqref{eq_intersectKl_Veronese72}, it follows that the decompositions $(0):l_1=(K,l_2)$ and $(0):l_2=(K,l_1)$ are Betti splittings by an application of Lemma \ref{lem_zeromap}. Hence by Corollary \ref{cor_conjpair}, $W$ is not absolutely Koszul. Since $\reg_R W=0$, by Theorem \ref{thm_Koszulmap}, $R=V_{7,2}$ is not absolutely Koszul.

\textsf{Case 2}: $(n,c)=(5,3)$.

In this case $N=35$, $B=k[a_1,\ldots,a_{30}]$. The Hilbert series of $U$ is $1+30t+45t^2+5t^3$. Denote $W=U/(a_{30})$; its Hilbert series is $1+29t+32t^2+t^3$. Denote $l_1=a_1+a_2-a_5+a_{24}+a_{29}, l_2=a_1+a_2-a_5-a_{24}-a_{29}$ and $K=(a_5a_{28},a_6a_{28},a_{15}a_{28})$. We have 
\begin{align*}
(0):l_1 &=K+(l_2),\\
(0):l_2 &=K+(l_1),\\
K\cap (l_1) &= K\cap (l_2)=(0).
\end{align*}
By Corollary \ref{cor_conjpair}, $W$ is not absolutely Koszul. Since $R$ is strongly Koszul, $\reg_R W=0$. Hence by Theorem \ref{thm_Koszulmap}, $R=V_{5,3}$ is not absolutely Koszul.

\textsf{Case 3:} $(n,c)=(5,4)$. 

In this case $N=70$, $B=k[a_1,\ldots,a_{65}]$. The Hilbert series of $U$ is $1+65t+155t^2+35t^3$. Denote $W=U/(a_{60},a_{61},\ldots,a_{65})$; its Hilbert series is $1+59t+63t^2+t^3$. Denote $l_1=a_1+a_5-a_{16}+a_{54}+a_{59}, l_2=a_1+a_5-a_{16}-a_{54}-a_{59}$ and $K=(a_2a_{58},a_6a_{58},a_{16}a_{58},a_{35}a_{58})$. We have
\begin{align*}
(0):l_1 &=K+(l_2),\\
(0):l_2 &=K+(l_1),\\
K\cap (l_1) &= K\cap (l_2)=(0).
\end{align*}
Hence by Corollary \ref{cor_conjpair}, $\lind_W (l_1)=\lind_W (l_2)=\infty$ and $W$ is not absolutely Koszul. Since $R$ is strongly Koszul, $\reg_R W=0$. Hence by Theorem \ref{thm_Koszulmap}, $R=V_{5,4}$ is not absolutely Koszul.

\textsf{Case 4:} $(n,c)=(4,5)$. In this case $N=56$, $B=k[a_1,\ldots,a_{52}]$. The Hilbert series of $U$ is $1+52t+68t^2+4t^3$. Denote $W=U/(a_{52})$; its Hilbert series is $1+51t+52t^2+t^3$. Denote $l_1=a_1+a_2-a_{10}+a_{45}+a_{51}$, $l_2=a_1+a_2-a_{10}-a_{45}-a_{51}$ and $K=(a_{10}a_{48})$. Then there are equalities:
\begin{align*}
(0):l_1 &=K+(l_2),\\
(0):l_2 &=K+(l_1),\\
K\cap (l_1) &=K \cap (l_2) = (0).
\end{align*}
Hence thanks to Corollary \ref{cor_conjpair}, $W$ is not absolutely Koszul. Since $R$ is strongly Koszul, $\reg_R W=0$. Hence by Theorem \ref{thm_Koszulmap}, $R=V_{4,5}$ is not absolutely Koszul. The proof is concluded.
\end{proof}

Tempted by the theoretical and experimental results available thus far, we ask:
\begin{quest}
Is it true that the ring $k[x_1,x_2,x_3,x_4]^{(c)}$ is not absolutely Koszul for all $c\ge 6$, at least in characteristic $0$?
\end{quest}
If this is true, then by Lemma \ref{lem_retract},  all the absolutely Koszul Veronese rings $k[x_1,\ldots,x_n]^{(c)}$ are described by Proposition \ref{prop_Veronese_positive}.

While many Veronese subrings and Segre products are not absolutely Koszul, it is not clear which of them are good. Hence we ask:
\begin{quest}
Are Veronese subrings and Segre products of polynomial rings good in the sense of Roos?
\end{quest}

\subsection*{Acknowledgments}
The author is a Marie Curie fellow of the Istituto Nazionale di Alta Matematica. He would like to thank Aldo Conca, Jan-Erik Roos and Marilina Rossi for inspiring discussions.

\end{document}